\pdfoutput=1
\documentclass[a4paper, 11pt, leqno, final]{amsart}
\usepackage{amsmath, amssymb, amsfonts, amsthm, bm, overpic, graphicx}
\usepackage[oneside,DIV=12,pagesize=pdftex]{typearea}
\newcommand{\conv}{\operatorname{conv}}
\newcommand{\Ray}[1]{\overrightarrow{\rule{0pt}{1.25ex}#1}}

\newcommand{\norm}[1]{\left\lVert#1\right\rVert}
\newcommand{\numbersystem}[1]{\mathbb{#1}}

\newcommand{\R}{\numbersystem{R}}

\newcommand{\N}{\mathcal{N}}
\newcommand{\M}{\mathcal{M}}

\newcommand{\abs}[1]{\lvert#1\rvert}

\newcommand{\myangle}{\sphericalangle}
\newcommand{\ipr}[2]{\left\langle #1, #2 \right\rangle}

\theoremstyle{plain}
\newtheorem{theorem}{Theorem}
\newtheorem{proposition}{Proposition}
\newtheorem{lemma}{Lemma}
\newtheorem{corollary}{Corollary}

\begin{document}

\bibliographystyle{amsplain}

\title{Absorbing angles, Steiner minimal trees, and antipodality}
\thanks{Research supported by a grant from an agreement between the Deutsche For\-schungs\-ge\-mein\-schaft in Germany and the National Research Foundation in South Africa. 
}
\author{Horst Martini}
\address{Fakult\"at f\"ur Mathematik,
        Technische Universit\"at Chemnitz,
        D-09107 Chemnitz, Germany}
\email{martini@mathematik.tu-chemnitz.de}
\author{Konrad J. Swanepoel}
\address{Fakult\"at f\"ur Mathematik,
        Technische Universit\"at Chemnitz,
        D-09107 Chemnitz, Germany}
\email{konrad.swanepoel@gmail.com}
\author{P. Oloff de Wet}
\address{Department of Decision Sciences,
	University of South Africa, PO Box 392,
	UNISA 0003, South Africa}
\email{dwetpo@unisa.ac.za}
\maketitle

\begin{abstract}
We give a new proof that a star $\{op_i:i=1,\dots,k\}$ in a normed plane is a Steiner minimal tree of its vertices $\{o,p_1,\dots,p_k\}$ if and only if all angles formed by the edges at $o$ are absorbing [Swanepoel, Networks \textbf{36} (2000), 104--113].
The proof is more conceptual and simpler than the original one.

We also find a new sufficient condition for higher-dimensional normed spaces to share this characterization.
In particular, a star $\{op_i: i=1,\dots,k\}$ in any CL-space
is a Steiner minimal tree of its vertices $\{o,p_1,\dots,p_k\}$ if and only if all angles are absorbing, which in turn holds if and only if all distances between the normalizations $\frac{1}{\norm{p_i}}p_i$ equal $2$.
CL-spaces include the mixed $\ell_1$ and $\ell_\infty$ sum of finitely many copies of $\R$. 

\medskip
Keywords: Steiner minimal tree, absorbing angle, antipodality, face antipodality, Minkowski geometry.
\end{abstract}

\section{Introduction}
\subsection{Minkowski geometry}
Let $\M^d$ denote a $d$-dimensional normed space (or \emph{Minkowski space}) with origin $o$, i.e., $\R^d$ equipped with a norm $\norm{\cdot}$.
We call an $\M^2$ a \emph{Minkowski plane}.
Denote the \emph{unit ball} by $B=\{x\in\R^d:\norm{x}\leq 1\}$.
The \emph{dual} $\M^d_\ast$ of $\M^d$ is $\R^d$ equipped with the \emph{dual norm}
\[\norm{x}_\ast:=\max_{\norm{y}\leq 1}\ipr{x}{y},\]
where $\ipr{\cdot}{\cdot}$ denotes the inner product on $\R^d$.
The \emph{dual unit ball}
\[B_\ast=\{x\in\R^d:\ipr{x}{y}\leq 1\;\forall y\in B\}\]
 is also known as the \emph{polar body} of $B$.

For example, the $d$-dimensional Minkowski spaces $\ell_1^d$ and $\ell_\infty^d$ are duals of each other, where $\ell_1^d$ has the norm $\norm{(x_1,\dots,x_d)}_1:=\sum_{i=1}^d\abs{x_i}$ and $\ell_\infty^d$ has the norm $\norm{(x_1,\dots,x_d)}_\infty:=\max\{\abs{x_i}: i=1,\dots,d\}$.

A vector $x_\ast\in\M^d_\ast$ is \emph{dual} to $x\in\M^d$, $x\neq o$, if $\norm{x_\ast}_\ast=1$ and $\ipr{x_\ast}{x}=\norm{x}$, i.e., $x_\ast$ is a dual unit vector that attains its norm at $x$.
In this case the hyperplane $\{x\in\R^d:\ipr{x_\ast}{x}=1\}$ supports the unit ball at $\frac{1}{\norm{x}}x$.
Any hyperplane supporting the unit ball at $\frac{1}{\norm{x}}x$ is given in this way by some $x_\ast$ dual to $x$.
A unit vector $v\in\M^d$ is a \emph{regular direction} if there is only one hyperplane that supports $B$ at $v$.
Note that the norm function $f(x):=\norm{x}$ is differentiable at $p\neq o$ if and only if $\frac{1}{\norm{p}}p$ is a regular direction, and then the gradient $\nabla f(p)$ is the unique vector in $\M^d_\ast$ dual to $p$.

The \emph{exposed face} of the unit ball $B$ defined by a unit vector $a^\ast\in\M^d_\ast$ is
\[ [a^\ast] := \{a\in B: \ipr{a}{a^\ast}=1\} .\]
Similarly, a unit vector $a\in\M^d$ defines an exposed face $[a]^\ast$ of $B^\ast$.
If $B$ is a polytope then all faces are exposed, and each face $F$ of $B$ corresponds to a face $F^\ast$ of $B^\ast$ as follows:
\[ F^\ast := \{a^\ast\in B^\ast : \ipr{a}{a^\ast}=1\text{ for all } a\in F\}.\]

\subsection{Trees}
Let $S\subset\M^d$ be a finite, non-empty set of points.
A \emph{spanning tree} $T$ of $S$
is an acyclic connected graph with vertex set $S$.
Denote its edge set by $E(T)$.
A \emph{Steiner tree} $T$ of $S$ is a
spanning tree of some finite $V\subset \M^d$ such that $S\subseteq V$ and such that the degree of each vertex in $V\setminus S$ is at least $3$.
The vertices in $S$
are the \emph{terminals} of the Steiner tree $T$, and the
vertices in $V\setminus S$ the \emph{Steiner points} of $T$.
The \emph{length} of a tree $T$ in $\M^d$ is
\[\ell(T):=\sum_{xy\in E(T)}\norm{x-y}.\]
A \emph{Steiner minimal tree} (SMT) of $S$ is a Steiner tree of $S$ of smallest 
length.
The requirement that Steiner points have degree at least $3$ is for technical convenience, since Steiner points of degree at most $2$ can easily be eliminated using the triangle inequality without making the tree longer.
It is easily seen that the number of Steiner points is at most $\#S-2$.
It then follows by a simple compactness argument that any non-empty finite $S$ has a SMT.

A \emph{star} with center $s$ is a tree in which the vertex $s$ is joined to all other vertices.
If $s\in S$ has neighbors $s_1,\dots,s_k\in V$ in some SMT, then clearly the star joining $s$ to each $s_i$, $i=1,\dots,k$, is a SMT of $\{s,s_1,\dots,s_k\}$.
Thus, to characterize the neighborhoods of terminals in SMTs, it is sufficient to characterize SMTs which are stars with the center a terminal.
This is the intent of Theorem~\ref{planetheorem} below.

\subsection{Angles}
An angle $\myangle x_1x_0x_2$ in $\M^d$ is \emph{absorbing} if
the function \[x\mapsto\norm{x-x_0}+\norm{x-x_1}+\norm{x-x_2}\]
attains its minimum at $x_0$.
Thus $\myangle x_1x_0x_2$ is absorbing if and only if the star $\{x_0x_1,x_0x_2\}$ is a SMT of $\{x_0,x_1,x_2\}$.

\begin{lemma}\label{absorbing}
Let $a$ and $b$ be unit vectors in $\M^d$.
Then the following are equivalent:
\begin{align}
\label{11} &\text{$\myangle aob$ is absorbing.}\\
\label{22} &\text{There exist unit vectors $a^\ast$ and $b^\ast$ in $\M^d_\ast$ such that}\\
\notag &\text{ $\ipr{a^\ast}{a}=\ipr{b^\ast}{b}=1$ and $\norm{a^\ast+b^\ast}^\ast\leq 1$.}\\
\label{33} & \text{The exposed faces $[a]^\ast$ and $[-b]^\ast=-[b]^\ast$ of the dual unit ball}\\
\notag &\text{ are at distance $\leq 1$.}
\end{align}
\end{lemma}
\begin{proof}
\eqref{11}$\Leftrightarrow$\eqref{22} is part of Lemma~5.4 in \cite{MSW}.
\eqref{22}$\Leftrightarrow$\eqref{33} is trivial.
\end{proof}

In particular, this is a property of the angle alone:
\begin{corollary}
If $y_i$ is a point on the ray $\Ray{x_ox_i}$, $i=1,2$, then $\{x_0x_1,x_0x_2\}$ is a SMT of $\{x_0,x_1,x_2\}$ if and only if $\{x_0y_1,x_0y_2\}$ is a SMT of $\{x_0,y_1,y_2\}$. \cite[Proposition~3.3]{MSW}.

Furthermore, an angle containing an absorbing angle is itself absorbing.
\end{corollary}
See also Proposition~3.3 and Lemma~5.4 of \cite{MSW}.

\subsection{The planar case}
We are now able to formulate the first result.
In any Minkowski space, all angles made by two incident edges of a SMT are clearly absorbing.
(For angles in a minimal spanning tree even more is true \cite{MS-AML}.)
Remarkably, as shown in \cite{Sw}, for a Minkowski plane the condition that all angles are absorbing is also sufficient for a star to be a SMT of its vertices.
\begin{theorem}\label{planetheorem}
Let $p_1,\dots,p_k\neq o$ be points in a Minkowski plane $\M^2$.
Then the star joining each $p_i$ to $o$ is a SMT of $\{o,p_1,\dots,p_k\}$ if (and only if) all angles $\myangle p_i o p_j$, $i\neq j$, are absorbing.
\end{theorem}
This result is used in \cite{Sw} to show that the maximum degree of a vertex in a SMT in a Minkowski plane is $6$, with equality only if the unit ball is an affine regular hexagon; for all other planes the maximum is $4$ if there exist supplementary absorbing angles, and $3$ otherwise.
The proof given in \cite{Sw} employs a long case analysis.
The new proof presented in Section~\ref{planeproof} is more conceptual.

\subsection{Antipodality and higher dimensions}
Theorem~\ref{planetheorem} does not hold anymore in Minkowski spaces of dimension at least $3$.
For example, let the unit ball be the projection of a $(d+1)$-cube along a diagonal. 
(When $d=3$, this is the rhombic dodecahedron.)
In this Minkowski space, the star joining $o$ to all $2^{d+1}-2$ vertices of the unit ball is not a SMT of these vertices if $d\geq 3$, despite all the angles being absorbing \cite{Sw-DCG}.
However, Theorem~\ref{planetheorem} extends to both $\ell_1^d$ and $\ell_\infty^d$.
Our second result is a generalization of this fact.

We first introduce some more notions, involving antipodality.
Two boundary points of the unit ball $B$ are \emph{antipodal} if there exist distinct parallel hyperplanes supporting the two points.
Equivalently, unit vectors $a$ and $b$ are antipodal if and only if $\norm{a-b}=2$.
\begin{lemma}
 If $a$ and $b$ are antipodal unit vectors in a Minkowski space, then $\myangle aob$ is an absorbing angle.
\end{lemma}
\begin{proof}
 Since the union of the segments $oa$ and $ob$ form a shortest path from $a$ to $b$, these two segments form a SMT of $\{o,a,b\}$, hence $\myangle aob$ is absorbing.
\end{proof}
The converse of the above lemma is not necessarily true, as the Euclidean norm shows.
We call the unit ball of a Minkowski space \emph{Steiner antipodal} if two points $a$ and $b$ on the boundary of the unit ball are antipodal whenever $\myangle aob$ is absorbing.

\begin{theorem}\label{steinerantipodal}
Consider the following properties of a set $\{p_1,\dots,p_k\}$ of unit vectors in a Minkowski space $\M^d$.
\begin{align}
\label{1} &\text{All angles $\myangle p_iop_j$ are absorbing.}\\
\label{2} &\text{All distances $\norm{p_i-p_j}=2$.}\\
\label{3} &\text{The star $\bigcup_{i=1}^k[o,p_i]$ is a SMT of $\{p_1,\dots,p_k\}$.}\\
\label{4} &\text{The star $\bigcup_{i=1}^k[o,p_i]$ is a SMT of $\{o,p_1,\dots,p_k\}$.}
\end{align}
Then the implications \eqref{2}$\Rightarrow$\eqref{3}$\Rightarrow$\eqref{4}$\Rightarrow$\eqref{1} hold.
Furthermore, \eqref{1} to \eqref{4} are equivalent if, and only if, the norm is Steiner antipodal.
\end{theorem}
\begin{proof}
\eqref{3}$\Rightarrow$\eqref{4} is true, since all Steiner trees of $\{o,p_1,\dots,p_k\}$ are also Steiner trees of $\{p_1,\dots,p_k\}$, by considering $o$ to be a Steiner point.

\eqref{4}$\Rightarrow$\eqref{1} holds, since if any star $[o,p_i]\cup[o,p_j]$ can be shortened, it would also shorten the star $\bigcup_{i=1}^k[o,p_i]$.

The implication \eqref{1}$\Rightarrow$\eqref{2} is equivalent to the definition of Steiner antipodality.

This leaves \eqref{2}$\Rightarrow$\eqref{3}.
Note that the given star is a Steiner tree of length $k$.
It is sufficient to show that all Steiner trees have length $\geq k$.
However, note that the open unit balls centered at the $p_i$ are pairwise disjoint, since $\norm{p_i-p_j}=2$ for distinct $i\neq j$.
Any Steiner tree will have to join each $p_i$ to the boundary of the unit ball with centre $p_i$.
The part of the Steiner tree inside this ball must therefore have length at least $1$.
It follows that the length of any Steiner tree must be at least $k$.
\end{proof}

In order to apply this result, we need a characterization of Steiner antipodal norms in terms of duality.

\begin{proposition}\label{steinerchar}
The following are equivalent in any Minkowski space $\M^d$:
\begin{align}
\label{a} &\text{The norm is Steiner antipodal.}\\
\label{c} &\text{The unit ball is a polytope and any two disjoint faces of}\\
\notag&\text{the dual unit ball are at distance $>1$.}
\end{align}
\end{proposition}
\begin{proof}
\eqref{a}$\Leftarrow$\eqref{c} is immediate from the definition of Steiner antipodality and Lemma~\ref{absorbing}.
\eqref{a}$\Rightarrow$\eqref{c} follows upon noting that if a convex body is not a polytope, then there are disjoint exposed faces that are arbitrarily close to each other. 
\end{proof}

A Minkowski space is a \emph{CL-space} if for every maximal proper face $F$ of the unit ball $B$ we have $B=\conv(F\cup(-F))$.
It is easily seen from finite dimensionality that the unit ball of a CL-space is a polytope.
CL-spaces were introduced by R. E. Fullerton (see \cite{Reisner}), although the notion has been studied before by Hanner \cite{Hanner}, who proved that the unit balls of CL-spaces are $\{0,1\}$-polytopes.
McGregor \cite{McGregor} showed that CL-spaces are exactly those spaces with numerical index $1$.
What is important for our purposes is that CL-spaces turn out to be Steiner antipodal.

Hanner \cite{Hanner} identified an important subclass of CL-spaces, namely those that can be built up from the one-dimensional space $\R$ using $\ell_1$-sums and $\ell_\infty$-sums.
For two Minkowski spaces $\M$ and $\N$ of dimension $d$ and $e$ we define their \emph{$\ell_1$-sum} $M\oplus_1\N$ and \emph{$\ell_\infty$-sum} $M\oplus_\infty\N$ to be the Minkowski spaces on $\R^{d+e}$ with norms $\norm{(x,y)}_1=\norm{x}+\norm{y}$ and $\norm{(x,y)}_\infty=\max\{\norm{x},\norm{y}\}$.
Note that the unit ball of $\M\oplus_1\N$ is the convex hull of the unit ball of $\M$ when embedded as $\M\oplus\{o\}$ and the unit ball of $\N$ when embedded as $\{o\}\oplus\N$.
The unit ball of $\M\oplus_\infty\N$ is the Cartesian product of the unit balls of $\M$ and $\N$.
The unit balls of these spaces are called \emph{Hanner polytopes}.
We thus introduce the name \emph{Hanner space} for these spaces.
For more information see \cite{Hanner, Hansen-Lima, Hansen, Reisner}.

We summarize the above discussion as follows.

\begin{proposition}
All Hanner spaces are CL-spaces.
All CL-spaces are Steiner antipodal.
\end{proposition}

\begin{proof}
It is clear and well-known that Hanner spaces are CL-spaces (see, e.g., \cite{Reisner}).

It is also well-known that the dual of a CL-space is a CL-space as well \cite{McGregor}.
To prove the second part of the proposition, it is by Proposition~\ref{steinerchar} sufficient to show that any two disjoint faces $F$ and $G$ of the unit ball $B$ are at distance $>1$.
Suppose that $F$ is contained in the facet $F'$.
Then all vertices of $B$ disjoint from $F$ must lie in the opposite facet $-F'$.
It follows that $G\subseteq -F'$, and $F$ and $G$ are therefore at distance $2$.
\end{proof} 
 
\section{Proof of Theorem~\ref{planetheorem}}\label{planeproof}

\begin{lemma}\label{l1}
Let $\ell$ be a line passing through a Steiner point $s$ of a SMT $T$ in a Minkowski plane $\M^2$.
Assume that $\ell$ is parallel to a regular direction.
Then $T$ has edges incident to $s$ in both open half planes bounded by $\ell$.
\end{lemma}
\begin{proof}
Without any assumption on $\ell$, the edges incident to $s$ cannot all lie in the same open half plane bounded by $\ell$.
Indeed, such a tree can be shortened as follows (Fig.~\ref{fig1}).
\begin{figure}
\begin{center}
\begin{overpic}[scale=0.5, angle=-90]{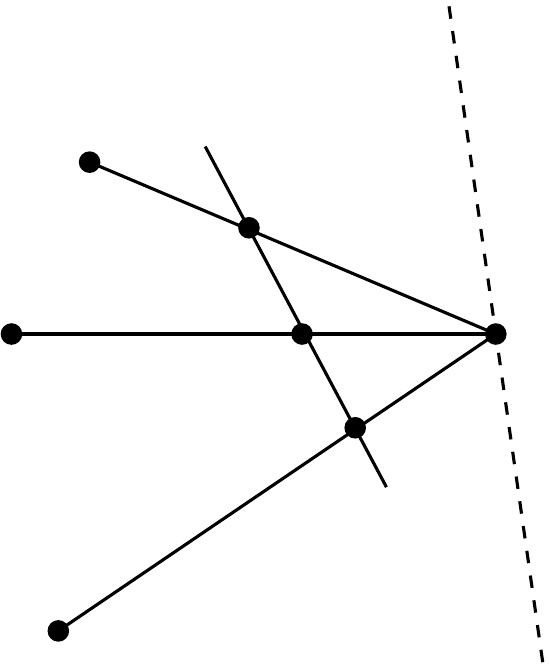}
\put(-7,73){$p_1$}
\put(35,80){$p_2$}
\put(81,71){$p_k$}
\put(22,29){$q_1$}
\put(38,40){$q_2$}
\put(68,37){$q_k$}
\put(80,51){$m$}
\put(90,16){$\ell$}
\put(45,-2){$s$}
\end{overpic}
\caption{}\label{fig1}
\end{center}
\end{figure}
Let some line $m$ intersect the interior of each edge $sp_i$, $i=1,\dots,k$, in $q_i$, say.
Remove edges $sp_1$, $sp_k$, and $sq_i$, $i=2,\dots,k-1$, and add edges $p_1q_2$, $q_iq_{i+1}$, $2\leq i\leq k-2$, and $q_{k-1}p_k$, to obtain a new Steiner tree $T'$, without the Steiner point $s$, but with new Steiner points $q_i$, $2\leq i\leq k-1$.
By the triangle inequality, $\ell(T)-\ell(T')\geq\sum_{i=2}^{k-1}\norm{sq_i}>0$, contradicting the minimality of $T$.

We now assume that $\ell$ is parallel to a regular direction.
It is sufficient to show that $sp_1$ and $sp_k$ cannot be opposite edges both on $\ell$, with all other $sp_i$, $2\leq i\leq k-1$, on the same side of $\ell$ (Fig.~\ref{fig2}).
\begin{figure}
\begin{center}
\begin{overpic}[scale=0.5, angle=-90]{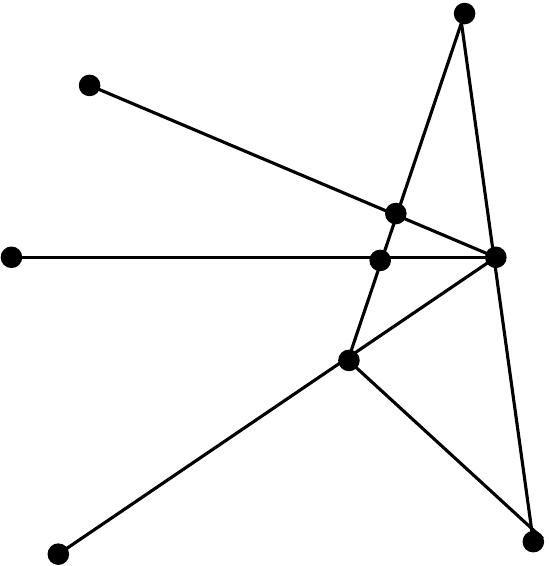}
\put(-2,10){$p_1$}
\put(103,16){$p_k$}
\put(-14,83){$p_2$}
\put(38,90){$p_3$}
\put(90,79){$p_{k-1}$}
\put(20,32){$s_2$}
\put(42,36){$s_3$}
\put(67,29){$s_{k-1}$}
\put(50,-2){$s$}
\end{overpic}
\caption{}\label{fig2}
\end{center}
\end{figure}
Let $s_2$ be a variable point on $sp_2$ with $\norm{s_2-s}$ small.
Denote the intersection of $s_2p_k$ and $sp_i$ by $s_i$, $i=3,\dots,k-1$.
Change the Steiner tree $T$ as follows.
Remove edges $p_1s$, $p_ks$ and $s_is$, $i=2,\dots,k-1$, and add edges $p_1s_2$ and $s_2p_k$.
This removes the Steiner point $s$ and introduces new Steiner points $s_2,\dots,s_{k-1}$.
Denoting the new tree by $T'$, it follows that the length changes by
\begin{align*}
& \ell(T')-\ell(T)\\
=& \norm{s_2-p_1}+\norm{s_2-p_k}-\norm{s-p_1}-\norm{s-p_k}-\sum_{i=2}^{k-1}\norm{s-s_i}\\
\leq&\norm{(s-p_1)+(s_2-s)}-\norm{s-p_1}\\
&+\norm{(s-p_k)+(s_2-s)}-\norm{s-p_k}-\norm{s_2-s}.
\end{align*}
Since $s-p_1$ and $s-p_k$ are parallel to a regular direction, the norm is differentiable at both points, i.e.,
\[ \lim_{s_2\to s}\frac{\norm{(s-p_1)+(s_2-s)}-\norm{s-p_1}}{\norm{s_2-s}} = \ipr{u_\ast}{s_2-s}\]
and
\[ \lim_{s_2\to s}\frac{\norm{(s-p_k)+(s_2-s)}-\norm{s-p_k}}{\norm{s_2-s}} = \ipr{-u_\ast}{s_2-s}.\]
(Since $s-p_1$ and $s-p_k$ are in opposite directions, their duals are opposite in sign.)
It follows that
\begin{align*}
\ell(T')-\ell(T) &\leq \ipr{u_\ast}{s_2-s}+\ipr{-u_\ast}{s_2-s}+o(\norm{s_2-s})-\norm{s_2-s}\\
&=o(\norm{s_2-s})-\norm{s_2-s},
\end{align*}
which is negative if $\norm{s_2-s}$ is sufficiently small.
Then $\ell(T')<\ell(T)$, a contradiction.
\end{proof}

\begin{proof}[Proof of Theorem~\ref{planetheorem}]
Without loss of generality, the segments $op_1,\dots, op_k$ are ordered around $o$.
Assume that all angles $\myangle p_i o p_j$ are absorbing.
We start off with an arbitrary SMT of $\{o,p_1,\dots,p_k\}$ and modify it in two steps without increasing the length.
In Step~1 we eliminate all Steiner points in the interiors of the angles $\myangle p_iop_{i+1}$.
In Step~2 we eliminate all edges between vertices on different rays $\Ray{op_i}$.
The edges of the final SMT are then all contained in the the union of the segments $op_i$, $i=1,\dots,k$.
This tree cannot have Steiner points, and so has to be the star with centre $o$.
This concludes the proof.

\smallskip
\textbf{Step 1:} For each angle $\myangle p_iop_{i+1}$ (where we let $k+1\equiv 0$), choose a regular direction $r_i$ not contained in the (closed) angle.
Choose $\widetilde{p_i}\in\Ray{op_i}$ and $\widetilde{q_i}\in\Ray{op_{i+1}}$ such that $\widetilde{p_i}\widetilde{q_i}$ is parallel to $r_i$ (Fig.~\ref{fig3}).
\begin{figure}
\begin{center}
\begin{overpic}[scale=0.5]{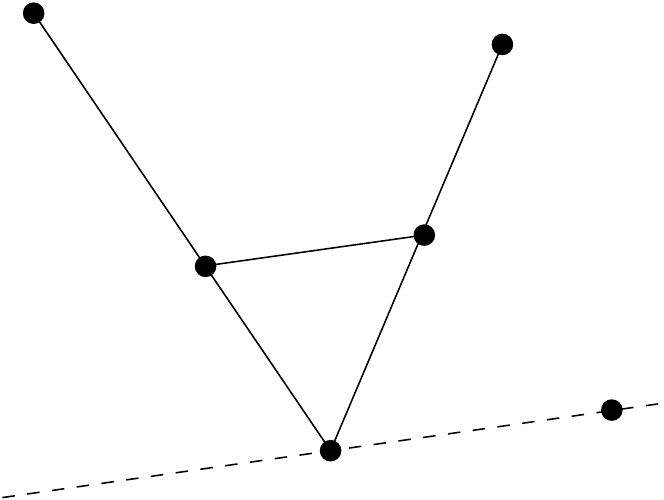}
\put(-7,73){$p_i$}
\put(81,69){$p_{i+1}$}
\put(18,33){$\widetilde{p_i}$}
\put(68,39){$\widetilde{q_i}$}
\put(90,19){$r_i$}
\put(45,-2){$o$}
\end{overpic}
\caption{}\label{fig3}
\end{center}
\end{figure}
For each point $s$ in the interior of $\myangle p_iop_j$, write $s=\alpha p+\beta q$ (uniquely, and then, moreover, $\alpha,\beta>0$) and define the \emph{measure} of $s$ to be
\[ \abs{s} := \alpha+\beta. \]
Define the \emph{measure} $\abs{T}$ of any Steiner tree $T$ of $\{o,p_1,\dots,p_k\}$ to be the sum of the measures of all Steiner points of $T$ not on any ray $\Ray{op_i}$.
Let
\[ \mu =\inf\{\abs{T}: \text{ $T$ is a SMT of $\{o,p_1,\dots,p_k\}$}\}.\]
Let $T_n$ be a sequence of SMTs of $\{o,p_1,\dots,p_k\}$ with $\lim_{n\to\infty}\abs{T_n}=\mu$.
Since there are only finitely many combinatorial types of Steiner trees on a set of $k+1$ points, we may, by passing to a subsequence, assume without loss of generality that all $T_n$ have the same combinatorial type with Steiner points $s_1^{(n)},\dots,s_m^{(n)}$, say.
By taking further subsequences, we may assume that each sequence of Steiner points converge, say $s_i^{(n)}\to s_i$, $i=1,\dots,m$.
In the limit we obtain a Steiner tree $T_0$ with $\ell(T_0)=\lim_{n\to\infty}\ell(T_n)$, hence $T_0$ is a SMT.
Also, $\abs{T_0}\leq\lim_{n\to\infty}\abs{T_n}$, since the measure of a Steiner point is continuous in the interior of an angle, hence $\lim_{n\to\infty}\abs{s_i^{(n)}}=\abs{s_i}$ if $s_i$ is still in the interior of the same angle, otherwise $\lim_{n\to\infty}\abs{s_i^{(n)}}\geq\abs{s_i}=0$ if $s_i$ is on one of the rays $\Ray{op_j}$.
Therefore, $\abs{T_0}=\mu$.
It remains to show that $\mu=0$, since this will imply that $T_0$ does not have any Steiner point in the interior of an angle.

Suppose that $\mu>0$.
We obtain a contradiction by constructing a SMT $T'$ with $\abs{T'}<\mu$.
Let $s$ be a Steiner point of $T_0$ in the interior of $\myangle p_iop_{i+1}$, say (Fig.~\ref{fig4}(a)).
\begin{figure}
\begin{center}
\begin{overpic}[scale=0.7, angle=-90]{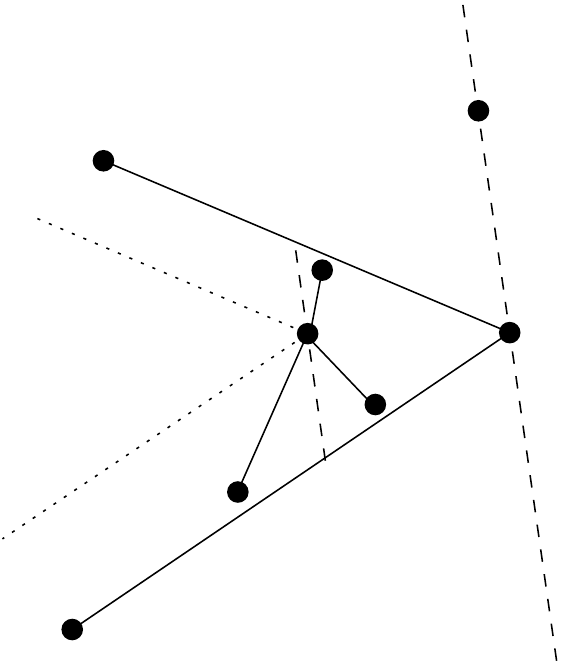}
\put(-4,73){$p_i$}
\put(79,69){$p_{i+1}$}
\put(85,16){$r_i$}
\put(47,0){$o$}
\put(53,40){$s$}
\put(31,37){$\ell$}
\put(41,23){$x_1$}
\put(63,33){$x_2$}
\put(25,52){$x_3$}
\put(45,-25){(a)}
\end{overpic}
\hspace{1cm}
\begin{overpic}[scale=0.7, angle=-90]{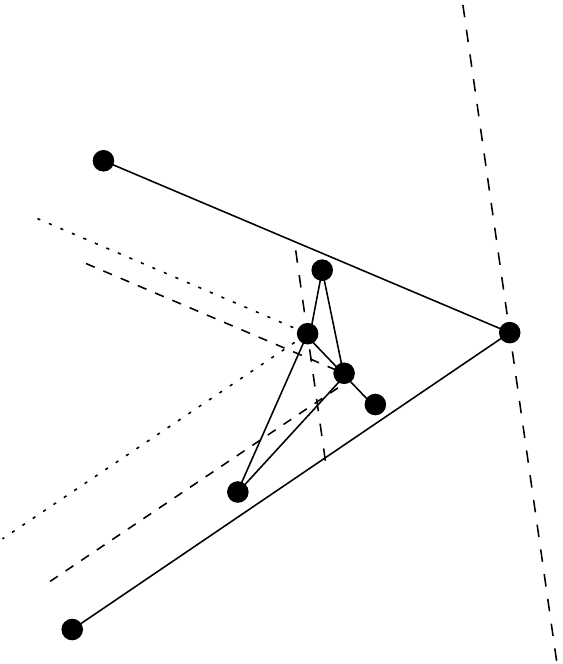}
\put(-4,73){$p_i$}
\put(79,69){$p_{i+1}$}
\put(47,0){$o$}
\put(29,23){$x_1$}
\put(63,33){$x_2$}
\put(13,46){$x_3$}
\put(53,40){$s$}
\put(45,26){$s'$}
\put(45,-25){(b)}
\end{overpic}
\vspace{1.5cm}
\caption{}\label{fig4}
\end{center}
\end{figure}
Without loss of generality there is no point of $T_0$ in the translated angle $s+\myangle p_iop_{i+1}$, since such a point is necessarily another Steiner point $s'$ and we may then repeatedly choose a new Steiner point $s''$ in $s'+\myangle p_iop_{i+1}$, until this procedure halts.

Let $\ell$ be the line through $s$ parallel to $r_i$.
The points  on $\ell$ in the interior of $\myangle p_iop_{i+1}$ all have the same measure, and the points on the same side of $\ell$ as $o$ have smaller measure.
By Lemma~\ref{l1} there is an edge $sx_1$ incident to $s$ on the same side of $\ell$ as $o$.
There are at least two more edges $sx_2$ and $sx_3$.
Since not all edges are in an open half plane bounded by a line through $s$, we may choose $x_2$ and $x_3$ such that the angle $\myangle x_2sx_3$ contains the translated angle $s+\myangle p_1op_2$ in its interior (with $s$ excluded).
It follows that there is a point $s'$ on $sx_1$ sufficiently close to $s$ such that $\myangle x_2s'x_3$ contains the translate $s'+\myangle p_1op_2$, and so is still absorbing (Figure~\ref{fig4}(b)).
We may therefore replace the edges $sx_2$, $sx_3$ and $ss'$ by $s'x_2$ and $s'x_3$ without lengthening $T_0$, to obtain a new SMT $T'$.
However, $\abs{s'}<\abs{s}$, hence $\abs{T'}<\abs{T}=\mu$, which gives the required contradiction.

\smallskip
\textbf{Step 2:}
Note that for any absorbing angle $\myangle p_i o p_j$,
\[\norm{p_i-o}+\norm{p_j-o}+\norm{o-o}\leq\norm{p_i-p_j}+\norm{p_j-p_j}+\norm{o-p_j},\]
i.e., $\norm{p_i-p_j}\geq\norm{p_i}$.

Suppose that the SMT $T$ has an edge between two points on different segments, say between $q_i$ on $op_i$ and $q_j$ on $op_j$.
Without loss of generality, the unique path in $T$ from $o$ to $q_i$ passes through $q_j$ (otherwise interchange $q_i$ and $q_j$).
Since $\myangle q_i o q_j$ is absorbing, $\norm{q_i-q_j}\geq\norm{q_i}$.
We can then replace the edge $q_i q_j$ by $oq_j$, without losing connectivity and without lengthening $T$.
This process may be repeated until all edges are on the segments $op_i$, which finishes Step 2.
\end{proof}

\end{document}